\documentclass[11pt]{article}
\usepackage{amssymb,amsmath,amsthm}
\usepackage{enumerate}
\usepackage{pgf,tikz}
\usepackage{mathrsfs}
\usepackage{rotating}
\usepackage{pdflscape}
\usetikzlibrary{arrows}
\usepackage{hyperref}
\usepackage{multirow}
\usepackage{listings}

\definecolor{codegreen}{rgb}{0,0.6,0}
\definecolor{codegray}{rgb}{0.5,0.5,0.5}
\definecolor{codepurple}{rgb}{0.58,0,0.82}
\definecolor{backcolour}{rgb}{0.95,0.95,0.92}

\lstdefinestyle{pythonstyle}{
    backgroundcolor=\color{backcolour},   
    commentstyle=\color{codegreen},
    keywordstyle=\color{magenta},
    numberstyle=\tiny\color{codegray},
    stringstyle=\color{codepurple},
    basicstyle=\footnotesize,
    breakatwhitespace=false,         
    breaklines=true,                 
    captionpos=b,                    
    keepspaces=true,                 
    numbers=left,                    
    numbersep=5pt,                  
    showspaces=false,                
    showstringspaces=false,
    showtabs=false,                  
    tabsize=2
}

\lstdefinestyle{outputstyle}{
    backgroundcolor=\color{backcolour},   
    numberstyle=\tiny\color{codegray},
    stringstyle=\color{codepurple},
    basicstyle=\footnotesize,
    breakatwhitespace=false,         
    breaklines=true,                 
    captionpos=b,                    
    keepspaces=true,                 
    numbers=left,                    
    numbersep=5pt,                  
    showspaces=false,                
    showstringspaces=false,
    showtabs=false,                  
    tabsize=2
}

\let\oldmarginpar\marginpar
\renewcommand\marginpar[1]{\-\oldmarginpar[\raggedleft\footnotesize #1]%
{\raggedright\footnotesize #1}}

\newtheorem{theorem}{Theorem}
\newtheorem{corollary}[theorem]{Corollary}
\newtheorem{obs}[theorem]{Observation}
\newtheorem{lemma}[theorem]{Lemma}
\newtheorem{proposition}[theorem]{Proposition}
\newtheorem{conjecture}[theorem]{Conjecture}

\newtheorem{example}[theorem]{Example}

\numberwithin{equation}{section}

\newcommand{\john}[1]{{\color{red} \sf John: [#1]}}

\newcommand{\note}[1]{{\color{violet} \sf Note: [#1]}}

\newcommand{\vanish}[1]{}

\usepackage{verbatim}

\begin{document}

\lstset{language = Python, frame = single, style = pythonstyle}

\title{New Perspectives on Neighborhood-Prime Labelings of Graphs}

\author{
John Asplund\\
{\small Department of Technology and Mathematics,
Dalton State College,} \\
{\small Dalton, GA 30720, USA} \\
{\small jasplund@daltonstate.edu}\\
\\
N. Bradley Fox\\
{\small Department of Mathematics and Statistics, Austin Peay State University} \\
{\small Clarksville, TN 37044} \\
{\small foxb@apsu.edu}\\
\\
Arran Hamm\\
{\small Department of Mathematics, Winthrop University} \\
{\small Rock Hill, SC 29733} \\
{\small hamma@winthrop.edu}
 }

\date{}
\maketitle

\begin{abstract}
Neighborhood-prime labeling is a variation of prime labeling. A labeling $f:V(G) \to [|V(G)|]$ is a neighborhood-prime labeling if for each vertex $v\in V(G)$ with degree greater than $1$, the greatest common divisor of the set of labels in the neighborhood of $v$ is $1$. In this paper, we introduce techniques for finding neighborhood-prime labelings based on the Hamiltonicity of the graph, by using conditions on possible degrees of vertices, and by examining a neighborhood graph.  In particular, classes of graphs shown to be neighborhood-prime include all generalized Petersen graphs, grid graphs of any size, and lobsters given restrictions on the degree of the vertices. In addition, we show that almost all graphs and almost all regular graphs have are neighborhood-prime, and we find all graphs of order $10$ or less that have a neighborhood-prime labeling. We give several conjectures for future work in this area. 
\end{abstract}

\section{Introduction}
Prime labeling is a type of graph labeling developed by Roger Entriger that was first formally introduced by Tout, Dabboucy, and Howalla~\cite{TDH}. We define $[n]:=\{1,\ldots,n\}$ where $n$ is a positive integer. 
Given a simple graph~$G$ of order $n$, a \textit{prime labeling} consists of labeling the vertices with integers from the set $[n]$ so that the labels of any pair of adjacent vertices are relatively prime.
Much of the work on prime labelings has occurred over the last two decades. One of the most studied conjectures, made by Entriger, is that all trees have a prime labeling. Several groups have worked on this conjecture, but this conjecture has yet to be completely resolved (see \cite{CFHK,HPT,P,S2}).
Almost all of the rest of the work on prime labelings has been focused on classes of graphs. 
Even for some grid graphs it is not known whether they have a prime labeling or not.
See the dynamic graph labeling survey~\cite{Gallian} by Gallian for an extensive list of graphs that have been shown to have prime labelings or that have been proven to not have such a labeling. 

The concept of a neighborhood-prime labeling, introduced by Patel and Shrimali~\cite{PS2}, is a variation of a prime labeling focusing on neighborhoods of vertices instead of endpoints of edges.  
We denote the neighborhood of a vertex $v$ in $G$ as $N_G(v)$, or simply $N(v)$ when the graph is clear.
A \textit{neighborhood-prime labeling} (NPL) of a graph $G$ of order $n$ is a bijective assignment of labels in $[n]$ to the vertices in $G$ such that for each $v\in V(G)$ of degree greater than $1$, the greatest common divisor ($\gcd$) of the labels of the vertices in $N(v)$ is $1$.  A graph that has a neighborhood-prime labeling is referred to as being \textit{neighborhood-prime}.  We usually use a bijective function--- $f: V(G)\rightarrow [n]$ such that $\gcd\{f(u) : u\in N(v)\}=1$ if $|N(v)|>1$ for all $v\in V(G)$---to represent this labeling.  To simplify the notation, we use $f(N(v))$ to denote $\{f(u) : u\in N(v)\}$.

Despite how much is known about graphs with prime labelings, very little is known about graphs with neighborhood-prime labelings. Only a few groups have worked on this problem. 
In their first paper on the topic, Patel and Shrimali~\cite{PS2} demonstrated the following graphs are neighborhood-prime for all sizes: paths, complete graphs, wheels, helms, and flowers.  Furthermore, cycles are neighborhood-prime if and only if their length $n$ satisfies $n\not\equiv 2\pmod{4}$.  Other graphs known to be neighborhood-prime include unions of paths and wheels, along with unions of cycles for certain lengths~\cite{PS3}.  
Recently Cloys and Fox~\cite{Fox} have shown that snake graphs are neighborhood-prime for particular lengths of the snake and of the polygons. In addition, Cloys and Fox showed numerous classes of trees to have an NPL such as caterpillars, spiders, firecrackers, and any tree that contains no degree $2$ vertices.  Based on these results, Cloys and Fox put forth a conjecture that all trees are neighborhood-prime, echoing the analogous conjecture made by Entriger for prime labelings.
Patel~\cite{Patel} has also studied generalized Petersen graphs although only a partial solution was obtained. 

In this paper, we investigate neighborhood-prime labelings from many angles. First, we show that any Hamiltonian graph of order $n\not\equiv 2\pmod{4}$ is neighborhood-prime. This is explained in detail in Section~\ref{ham_cycle} as well as other properties related to Hamiltonicity that guarantee a graph is neighborhood-prime. 
These properties are extended and applied to certain classes of graphs in Section~\ref{apps}, including generalized Petersen graphs and grid graphs. Of particular note, we are able to complete Patel's partial results to show all generalized Petersen graphs are neighborhood-prime. 
Section~\ref{small} focuses on graphs with ten or less vertices. In this section, we find all simple graphs that are neighborhood-prime and provide all such graphs that are not neighborhood-prime. 
Based on the list of graphs on ten or less vertices that are not neighborhood-prime, we conjecture that all graphs with minimum degree at least $3$ are neighborhood-prime. To further justify this conjecture, we show in Section~\ref{random} that random $d$-regular graphs $G_{n,d}$ are neighborhood-prime with high probability for $d\geq 3$ (where, as is usual, ``with high probability'' means with probability tending to one as $n$ tends to infinity). We also show that random graphs $G_{n,p}$ are neighborhood-prime w.h.p. when $p$ is large enough. 
In the final section, we explore another property that can guarantee that a graph is neighborhood-prime. This property creates a link between the rich literature of prime labelings and that of the neighborhood-prime labelings. Besides showing several classes of graphs and trees as neighborhood-prime, we show in Section~\ref{sec:trees} that as long as $n$ is large enough, all trees of order $n$ are neighborhood-prime. This mirrors the work done in \cite{HPT} to address Entriger's conjecture on trees.

\section{Hamilton Cycles}\label{ham_cycle}
Two powerful observations allow the creation of neighborhood-prime graphs from other neighborhood-prime graphs. First, observe that the graph resulting from attaching a new vertex to exactly one vertex that was of degree at least 2 in a graph with an NPL is neighborhood-prime. Second, if a graph $G$ is neighborhood-prime, then the graph formed by adding an edge in $G$ between two vertices of degree at least 2 is still neighborhood-prime. This second observation is rather powerful since the result by Patel and Shrimali \cite{PS2} that showed that a cycle of length $n$, denoted $C_n$, is neighborhood-prime for $n\not\equiv 2\pmod{4}$ can be extended to a stronger result.

\begin{theorem}\label{Hamiltonian}
Let $G$ be a graph of order $n$ such that $n\not\equiv 2\pmod{4}$.  If $G$ is Hamiltonian, then $G$ has a neighborhood-prime labeling.
\end{theorem}

To label the vertices of a graph with Hamilton cycle $C=(v_1,\ldots, v_n)$, we apply our second observation and use the following labeling function, which is a reformulation of the labeling used by Patel and Shrimali~\cite{PS2} for cycles:
\begin{equation}\label{cycle}
f(v_{i})=\begin{cases}
\text{  }\lfloor\frac{n}{2}\rfloor+\frac{i+1}{2} \hspace{.6cm}\text{if $i$ is odd}\\
\text{  }\frac{i}{2} \hspace{2cm}\text{if $i$ is even.}\\
\end{cases}
\end{equation}

Despite following directly from the existence of cycles that are neighborhood-prime, Theorem~\ref{Hamiltonian} is a strong result as it reveals a large class of graphs that are neighborhood-prime.  Many of the theorems in previous papers on neighborhood-prime labelings---such as Patel and Shrimali showing that helms, closed helms, and flowers are neighborhood-prime \cite{PS2} or Cloys and Fox demonstrating an NPL for the gear graph~\cite{Fox}---apply only to a single graph for any given order $n$.  
While some results---such as Patel and Shrimali's proof of unions of certain cycles being neighborhood-prime~\cite{PS3} or Cloys and Fox's examination of neighborhood-prime labelings of caterpillars---are much more general, the required structure of these graphs severely limits the number of graphs of a given order that are neighborhood-prime.  Theorem~\ref{Hamiltonian}, on the other hand, provides a large number of graphs that are neighborhood-prime as can be seen for the orders of $n\leq 12$ given below (number of Hamiltonian graphs for each $n$ provided by \textit{The On-Line Encyclopedia of Integer Sequences} published electronically at \url{https://oeis.org}, 2015, Sequence A003216).  

\begin{center}
	\begin{tabular}{|c|cccccccc|}
	\hline 
	$n$ & $3$ & $4$ & $5$ & $7$ & $8$ & $9$ & $11$ & $12$ \\
	\hline
	Num. of graphs with& \multirow{2}{*}{1} & \multirow{2}{*}{3} & \multirow{2}{*}{8} & \multirow{2}{*}{383} & \multirow{2}{*}{6196} & \multirow{2}{*}{177083} & \multirow{2}{*}{883156024} & \multirow{2}{*}{152522187830}\\
	NPLs by Thm~\ref{Hamiltonian} &  &  &  &  &  &  &  & \\
	\hline
	\end{tabular} 
\end{center}

An immediate consequence of Theorem~\ref{Hamiltonian} is that by adding vertices to a Hamiltonian graph in a particular manner ensures the graph is still neighborhood-prime. 

\begin{corollary} 
Let $H$ be neighborhood-prime with a Hamilton cycle $C=(v_0,\ldots,v_{n-1})$ of length $n\not\equiv 2\pmod{4}$. Let $G$ be formed from $H$ with $j$ additional vertices $\{u_1,\ldots,u_j\}$ such that for each $u_i$, $u_iv_{\ell_i},u_iv_{\ell_i+2}\in E(G)$ where the subscripts of $V(C)$ are calculated modulo $k$.  Then $G$ is neighborhood-prime. 
\end{corollary}
\begin{proof}
We can label $G$ using a labeling similar to Equation~\ref{cycle} by shifting the indices with $g(v_i)=f(v_{i+1})$, and assigning each $g(u_i)$ as any label in $\{n+1,\ldots, n+j\}$.  We only need to consider $\gcd\{f(N(u_i))\}$ since $N_H(v_i)\subseteq N_G(v_i)$.  Since $v_{\ell_i}$ and $v_{\ell_i+2}$ are either labeled by consecutive integers or one of them is labeled as $1$, $g$ is an NPL.
\end{proof}

Although cycles of order $2$ modulo 4 are not neighborhood-prime, the next couple of lemmas show that having even one additional edge may be enough to have a NPL.

\begin{lemma}\label{HamChord4k}
If the graph $G$ contains a Hamilton cycle $C$ and a chord that forms a cycle of length $4k$ for some positive $k\in \mathbb{Z}$ using only the chord and edges from $C$, then $G$ is neighborhood-prime.
\end{lemma}

\begin{proof}
By Theorem~\ref{Hamiltonian}, we may assume the order of $G$ is $n=2\ell$ where $\ell$ is an odd integer. Let $C=(v_1,\ldots,v_n)$. Then without loss of generality we can assume the chord described in the statement of this lemma is the edge $v_n v_{4k-1}$ for some positive $k\in\mathbb{Z}$.  

Applying Patel and Shrimali's labeling in Equation~\eqref{cycle} to a cycle of length equivalent to $2$ modulo $4$ results in a labeling that satisfies the neighborhood-prime condition for the neighborhood of every vertex except for $v_n$.  The labels in this neighborhood within the cycle $C$ would be $f(N_C(v_n))=f(\{v_1,v_{n-1}\})=\{n/2+1,n\}=\{\ell+1,2\ell\}$.  Notice that $\ell+1$ and $2\ell$ are both even.  In particular, these labels have only a common factor of $2$ since $\gcd\{\ell+1,2\ell\}=\gcd\{\ell+1,\ell-1\}$.  

With the additional neighbor $v_{4k-1}$ in $N_G(v_n)$, we have $f(\{v_1,v_{n-1}, v_{4k-1})\}\subseteq f(N_G(v_n))$, with the label on the third vertex being $f(v_{4k-1})=\ell+2k$.  Since this label on $v_{4k-1}$ is odd and $\gcd\{v_1,v_{n-1}\}=2$, we obtain $\gcd\{f(N(v_n))\}=1$, proving that the labeling $f$ is an NPL.
\end{proof}

\begin{lemma}\label{oddCycle}
If $G$ is Hamiltonian and contains an odd cycle then $G$ is neighborhood-prime. 
\end{lemma}

\begin{proof}
By Theorem~\ref{Hamiltonian}, we may assume that $n\equiv 2\pmod{4}$.
Since $G$ contains an odd cycle and there are an even number of vertices in $G$, there must be a chord in $G$ that forms an odd cycle of length $k$ using $k-1$ consecutive edges from the Hamilton cycle. 
be the Hamilton cycle so that $v_1v_{k}$ is the chord that forms the odd cycle $(v_1,\ldots,v_k)$.
Since $n\equiv 2\pmod{4}$, label the sequence of vertices as follows:
\begin{align*}
1,2,3,\ldots,\frac{n}{2} &\to v_1,v_3,v_5,\ldots,v_{n-1}\\
 \frac{n}{2}+1,\frac{n}{2}+2,\frac{n}{2}+3,\ldots,n &\to v_{k+1},v_{k+3},v_{k+5},\ldots, v_{n}, v_2, v_4,\ldots,v_{k-1}
\end{align*}
By design, every vertex has neighbors within the cycle that are labeled by consecutive integers or by the label $1$, with the exception of $v_k$.  The neighbors of $v_k$ on the cycle have even labels $\frac{n}{2}+1$ and $n$, but since the label on $v_1$ is $1$, the chord $v_1v_k$ ensures that $\gcd\{f(N(v_k))\}=1$, and thus $G$ is neighborhood-prime. 
\end{proof}

The previous two results show that the only possible Hamiltonian graphs that may not be neighborhood-prime are those with subgraphs of cycles with lengths that are only 2 modulo 4. We can go further by showing that the labeling used for the graphs in Theorem~\ref{Hamiltonian} can in some cases be used on graphs that are not Hamiltonian.  The \textit{circumference} of a graph is defined as the length of the longest cycle in that graph.  Hamiltonian graphs of order $n$ therefore have a circumference of $n$, but graphs with circumference of $n-1$ are also neighborhood-prime assuming certain restrictions on $n$. 

\begin{proposition}\label{circumference}
All connected graphs of order $n$ with $n\not\equiv 3\pmod{4}$ that have circumference $n-1$ are neighborhood-prime.
\end{proposition}
\begin{proof}
Consider an $(n-1)$-cycle $C=(v_1,\ldots,v_{n-1})$ on the graph in which we denote by $u$ the vertex in $G$ not included on $C$.  We label $C$ using the function found in Equation~\eqref{cycle} (with the length $n-1$ in place of $n$)
such that without loss of generality $v_2$ is adjacent to $u$, and then assign to $u$ the label $n$.  If $d(u)=1$, there's nothing to check; otherwise, the vertices on the cycle satisfy the neighborhood-prime condition because $n-1\not\equiv 2\pmod{4}$.  Since $f(v_2)=1$ with $v_2$ being selected to be in the neighborhood of $u$, $\gcd\{f(N(u))\}=1$, making the labeling neighborhood-prime.
\end{proof}

\section{Applications of Section~\ref{ham_cycle}}\label{apps}

\subsection{Generalized Petersen Graphs}

An interesting set of graphs that have been investigated for neighborhood-prime labelings that are often Hamiltonian are generalized Petersen graphs.  This graph, denoted by $GP(n,k)$ for $n\geq 2$ and $1\leq k<n/2$, consists of a vertex set $\{u_0,u_1,\ldots, u_{n-1},v_0,v_1,\ldots v_{n-1}\}$ and three types of edges: $u_i u_{i+1}$, $u_i v_i$, and $v_i v_{i+k}$ for each $i\in\mathbb{Z}_n=\{0,1,\ldots,n-1\}$ with the subscripts reduced modulo $n$.  
An example of $GP(12,3)$ is shown in Figure~\ref{GPgraph}, in which the NPL is found using the Hamiltonian cycle $(v_1, v_{10}, u_{10}, u_{11}, v_{11}, v_2, u_2, u_3, u_4, v_4, v_7, u_7, u_6, u_5, v_5, v_8, u_8, u_9, v_9, v_6, v_3, v_0, u_0, u_1)$.
Patel~\cite{Patel} demonstrated that $GP(n,k)$ is neighborhood-prime whenever $n$ and $k$ have a greatest common divisor of $1$, $2$, or $4$, and the special case of $GP(n,8)$.  Using our results from Section~\ref{ham_cycle}, we will prove the generalized Petersen graphs are all neighborhood-prime regardless of the common divisor of $n$ and $k$.

\begin{theorem} \label{gen_pet}
The generalized Petersen graph $GP(n,k)$ is neighborhood-prime for all $n$ and $k$.
\end{theorem}

\begin{figure}[htb]
\begin{center}
\includegraphics[scale=1]{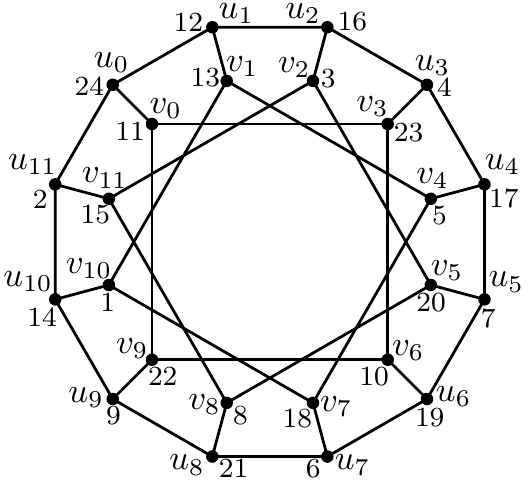}
\caption{The graph $GP(12,3)$ with a neighborhood-prime labeling}\label{GPgraph}
\end{center}
\end{figure}

It was shown by Alspach~\cite{Alspach} that $GP(n,k)$ is Hamiltonian for all $n$ and $k$ except for the following two cases:
\begin{itemize}
\item $n\equiv 5\pmod{6}$ with $k=2$ or $(n-1)/2$, which are known to be isomorphic graphs, or

\item $n\equiv 0\pmod{4}$ and $n\geq 8$ with $k=n/2$.
\end{itemize}

Since $|V(GP(n,k))|=2n$, we observe that $|V(GP(n,k))|\equiv 0\pmod{4}$ if $n$ is even and $2\pmod{4}$ when $n$ is odd.  Therefore, Theorem~\ref{Hamiltonian} would only apply when $n$ is even, thus we obtain the following corollary, which greatly expands upon Patel's conditions despite only applying to even $n$.

\begin{corollary}
The generalized Petersen graph $GP(n,k)$ is neighborhood-prime for all even $n$ and any $k$ except possibly when $n\equiv 0\pmod{4}$ and $n\geq 8$ with $k=n/2$.
\end{corollary}

The single remaining case of $GP(n,k)$ with $n$ being even is when  $n\equiv 0\pmod{4}$, $n\geq 8$, and $k=n/2$ in which case the graph is not Hamiltonian, hence Theorem~\ref{Hamiltonian} cannot apply.  Instead, the next theorem handles this special case with an explicit labeling, which can be seen in Figure~\ref{GPlastcase} for the specific case of $n=8$ and $k=4$.

\begin{figure}[htb]
\begin{center}
\includegraphics[scale=1.1]{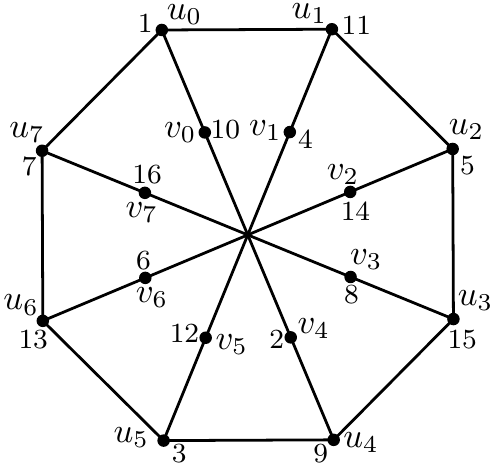}
\caption{The graph $GP(8,4)$ with a neighborhood-prime labeling}\label{GPlastcase}
\end{center}
\end{figure}

\begin{lemma}
The generalized Petersen graph $GP(n,n/2)$ is neighborhood-prime for all $n\geq 8$ with $n\equiv 0\pmod{4}$.
\end{lemma}
\begin{proof}
First, note that in $GP(n,n/2)$ the interior vertices $v_i$ are connected such that for each $i\in\mathbb{Z}_n$, $N(v_i)=\{u_i,v_{i+n/2}\}$ where indices are reduced modulo $n$.  We label the vertices with the function $f:V(G)\to \mathbb{Z}_{2n}$ as follows. 
\begin{align*}
f(u_{1+2t})&\equiv n+3+4t\pmod{2n} \;\;\;\; \text{ for } t\in\mathbb{Z}_{n/2}\\
f(u_{2t})&\equiv 1+4t\pmod{2n} \;\;\;\; \text{ for } t\in\mathbb{Z}_{n/2}\\
f(v_{2t})&\equiv n+2+4t\pmod{2n} \;\;\;\; \text{ for } t\in\mathbb{Z}_{n/2}\\
f(v_{1+2t})&\equiv 4+4t\pmod{2n} \;\;\;\; \text{ for } t\in\mathbb{Z}_{n/2}\\
\end{align*}

For each vertex $u_i$ with $i\in\mathbb{Z}_{n-1}$, we have $v_i, u_{i+1}\in N(u_i)$, and these vertices are labeled by consecutive integers.  All interior vertices $v_i$ have a neighborhood of $N(v_i)=\{u_i,v_{i+n/2}\}$.  Similarly, these vertices are labeled by consecutive integers.  Finally, the vertex $u_{n-1}$ has $u_0$ in its neighborhood, which is labeled as $1$.  In all cases, the $\gcd$ of the labels in each neighborhood is $1$, proving the graph is neighborhood-prime.
\end{proof}

Now that all generalized Petersen graphs with $n$ being even have been shown to be neighborhood-prime, we next consider the odd cases.  The lone cases that are not Hamiltonian are when $n\equiv 5\pmod{6}$ with $k=2$ or $(n-1)/2$.  For the remaining cases, for which $|V(GP(n,k))|=2n\equiv 2\pmod{4}$ with $n$ being odd, Lemma~\ref{oddCycle} can be used to prove this case.
This is summarized in the following corollary. See Figure~\ref{GenPetOddCycle} for an example of $GP(9,3)$ being labeled using the Hamiltonian cycle $$(v_0, v_3, v_6, u_6, u_5, v_5, v_2, v_8, u_8, u_7, v_7, v_1, v_4, u_4, u_3, u_2, u_1, u_0)$$
and using the chord $v_0 v_6$ for the labeling described in Lemma~\ref{oddCycle}.  Generally, such a chord is guaranteed to exist since $(u_0, u_1,\ldots, u_{n-1})$ forms an odd cycle.

\begin{figure}[htb]
\begin{center}
\includegraphics[scale=1]{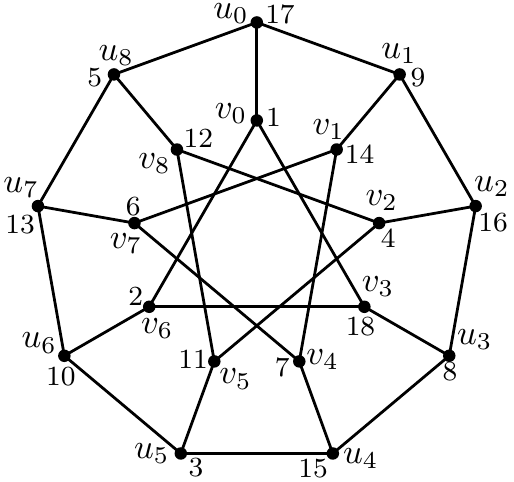}
\caption{The graph $GP(9,3)$ with a neighborhood-prime labeling}\label{GenPetOddCycle}
\end{center}
\end{figure}

\begin{corollary}
The generalized Petersen graph $GP(n,k)$ for odd $n$ is neighborhood-prime, except possibly when $n\equiv 5\pmod{6}$ and $k=2$ or $(n-1)/2$.
\end{corollary}

The last remaining cases of $n\equiv 5\pmod{6}$ with $k=2$ or $(n-1)/2$ are isomorphic graphs, and are already proven by Patel~\cite{Patel} to be neighborhood-prime since $n$ is odd, making the only common divisor of $n$ and $k=2$ be 1.  The previous three Corollaries and Lemmas, along with Patel's results, combine to prove the main result of this section, Theorem~\ref{gen_pet}.

\subsection{Grid Graphs}

We now move our attention to another class of graphs that are often Hamiltonian: grid graphs.  A grid graph is a Cartesian product of paths, $P_m\times P_n$, and is known to be Hamiltonian if either $m$ or $n$ is even, or both.  We will show that all grid graphs are neighborhood-prime using several results depending on the parity of $m$ and $n$. In the case of $|V(P_m\times P_n)|=m n\equiv 0\pmod{4}$, Theorem~\ref{Hamiltonian} directly proves the following.  See Figure~\ref{grids1} for an example of the graph $P_4\times P_5$ with its Hamilton cycle displayed with solid edges.

\begin{corollary}
The grid graph $P_m\times P_n$ is neighborhood-prime if $m$ and $n$ are both even or if one of $m$ or $n$ is equivalent to $0$ modulo $4$.  
\end{corollary}

\begin{figure}[htb]
\begin{center}
\includegraphics[scale=.85]{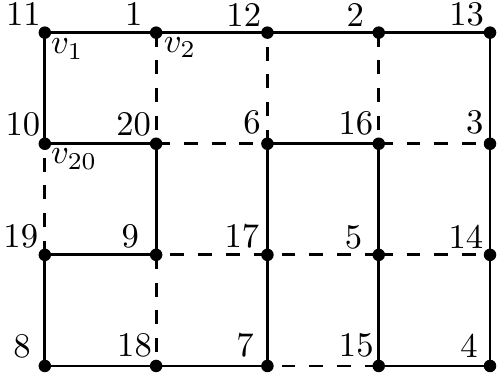}\hspace{1cm}\includegraphics[scale=.85]{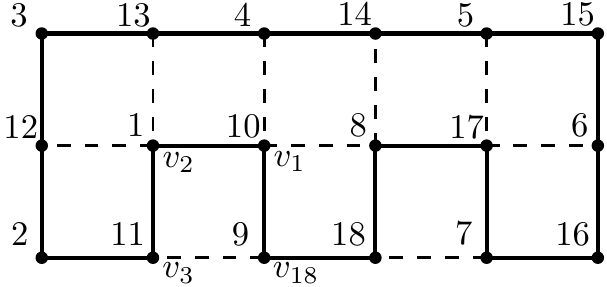}
\end{center}
\caption{The grid graphs $P_4\times P_5$ and $P_3\times P_6$}\label{grids1}
\end{figure}

Grid graphs that are Hamiltonian but are of order $n\equiv 2\pmod{4}$ naturally fit the criteria of Lemma~\ref{HamChord4k} with one of the squares within the grid viewed as the $4$-cycle, which proves the following case to be neighborhood-prime.  Figure~\ref{grids1} includes an example of $P_3\times P_6$ with an NPL where the chord between $v_{18}$ and $v_3$ allows the neighborhood of $v_{18}$ to have labels with a greatest common divisor of $1$.

\begin{corollary}
The grid graph $P_m\times P_n$ is neighborhood-prime if exactly one of $m$ or $n$ is odd and the other is equivalent to $2$ modulo $4$.
\end{corollary}

The grid graph $P_m\times P_n$, with $m$ and $n$ both being odd, is the only non-Hamiltonian case.  However, it has a circumference of $mn-1$, meaning we can apply Proposition~\ref{circumference} to prove the following case in which $mn\not\equiv 3\pmod{4}$ is neighborhood-prime.  See Figure~\ref{grids2} for an example of $P_3\times P_7$ with solid edges displaying the cycle of length $20$.

\begin{corollary}
The grid graph $P_m\times P_n$ is neighborhood-prime if $m,n\equiv 1\pmod{4}$ or $m,n\equiv 3\pmod{4}$.
\end{corollary}
\begin{proof}
In each case, the size of the vertex set is $m n\equiv 1\pmod{4}$.  Therefore, since the circumference of a grid graph with $m$ and $n$ both being odd is $m n-1$, we can apply Proposition~\ref{circumference} to demonstrate the graph is neighborhood-prime.
\end{proof}

\begin{figure}[htb]
\begin{center}
\includegraphics[scale=1]{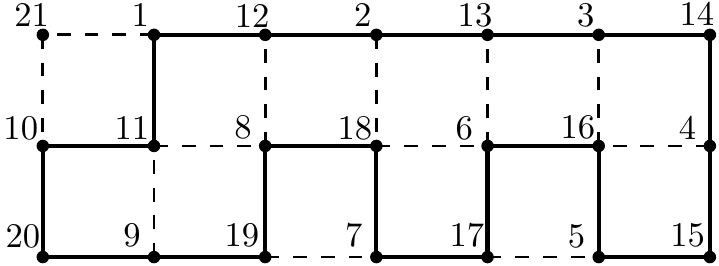}\hspace{.5cm}\includegraphics[scale=1]{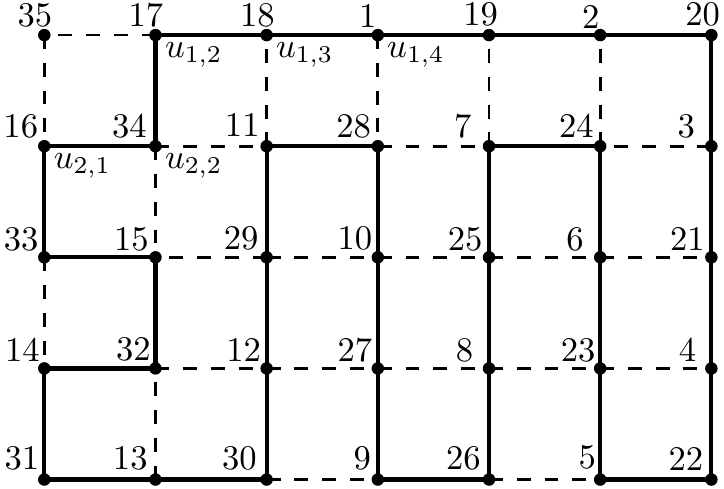}
\end{center}
\caption{The grid graphs $P_3\times P_7$ (left) and $P_5\times P_7$ (right)}\label{grids2}
\end{figure}

The final case that remains to show that all grid graphs are neighborhood-prime is when $m\equiv 1\pmod{4}$ and $n\equiv 3\pmod{4}$ (or vice versa).  This graph still has a circumference of $mn-1$, but $mn\equiv 3\pmod{4}$, so Proposition~\ref{circumference} cannot apply directly.  We introduce a more explicit labeling to handle this last case, as shown in the example of $P_5\times P_7$ in Figure~\ref{grids2}.

\begin{lemma}
The grid graph $P_m\times P_n$ is neighborhood-prime if $m\equiv 1\pmod{4}$ and $n\equiv 3\pmod{4}$.
\end{lemma}
\begin{proof}
Consider the vertices of this grid graph as $u_{i,j}$ with $1\leq i\leq m$ and $1\leq j\leq n$.  There exists a cycle of length $m n-1$ (missing vertex $u_{1,1}$) using the following sequence of vertices, as seen in the example in Figure~\ref{grids2}, 
$$C=(u_{1,3},u_{1,4},\ldots, u_{1,n}, u_{2,n},u_{3,n},\ldots, u_{m,n}, u_{m,n-1}, u_{m-1,n-1},\ldots u_{2,n-1}, u_{2,n-2}, u_{3,n-2},\ldots, u_{m,n-2},\ldots,$$
$$u_{m,3}, u_{m,2}, u_{m,1}, u_{m-1,1}, u_{m-1,2}, u_{m-2,2}, u_{m-2,1},\ldots, u_{3,1}, u_{2,1},u_{2,2}, u_{1,2}).$$
If we label $C$ using the standard cycle labeling from Equation~\eqref{cycle}, the labels on the neighbors $N_C(v)$ satisfy the neighborhood-prime labeling condition for all vertices except the last vertex $u_{1,2}$.  Two of the neighbors of $u_{1,2}$ are $u_{1,3}$ and $u_{2,2}$, which are labeled by $m n-1$ and $\lfloor(m n-1)/2\rfloor+1$, both of which are even.  However, we assign to the vertex $u_{1,1}$, which is left out of the cycle, the label $m n$.  Since $u_{1,1}\in N(u_{1,2})$, the neighborhood of $u_{1,2}$ now has consecutive labels. Also notice that the neighborhood of $u_{1,1}$ contains two vertices with consecutive labels, resulting in the labeling being an NPL.
\end{proof}

The culmination of this section is the following result which follows from all of the previous four Corollaries and Lemmas in this section.

\begin{theorem}
For all positive integers $m$ and $n$, $P_m\times P_n$ is neighborhood-prime.
\end{theorem}

We can extend the reach of Theorem~\ref{Hamiltonian} for higher dimensional grids as well since $P_{\ell}\times P_m\times P_n$ is Hamiltonian except when $\ell, m,$ and $n$ are all odd.  We also need to ensure that the cycle length avoids being equivalent to $2\pmod{4}$.

\begin{theorem}
The three-dimensional grid graph $P_{\ell}\times P_m\times P_n$ is neighborhood-prime if the number of vertices satisfies $\ell m n \equiv 0\pmod{4}$.
\end{theorem}

We believe that the results in this section can each be extended to show all higher dimensional grid graphs are neighborhood-prime.

\begin{conjecture}
The graph $P_{n_1}\times P_{n_2}\times \cdots \times P_{n_t}$ is neighborhood-prime.
\end{conjecture}

\subsection{Graphs Based on Restrictions}\label{restrict}

In this section, we use the results in Section~\ref{ham_cycle} to show that there are limits to the number of edges and vertices  a graph can have before we can guarantee the graph is neighborhood-prime. 

\begin{proposition}\label{ham_max_edges}
A Hamiltonian graph of order $n$ with $|E|> n\left\lfloor\frac{n-6}{8}\right\rfloor+n$ is neighborhood-prime.
\end{proposition}

\begin{proof}
Let $G$ be a Hamiltonian graph of order $n$ that is not neighborhood-prime. By Theorem~\ref{Hamiltonian}, we may assume that $n\equiv 2\pmod{4}$. By Lemmas~\ref{HamChord4k} and \ref{oddCycle}, all chords on a Hamiltonian cycle in $G$ must form cycles of lengths congruent to $2$ modulo $4$. 
Let a Hamilton cycle in $G$ be $C=(v_1,v_2,\ldots,v_n)$. Then $v_1$ can be adjacent to $v_6,v_{10},v_{14},\ldots, v_{n-4}$. There are $\frac{n-6}{4}$ such chords in $C$ for each vertex in $V(G)$, but half of these chords will be duplicates. Hence there are at most $n\frac{n-6}{8}$ chords in $C$ in addition to the $n$ edges from $C$ itself, showing any more edges would form a graph that is neighborhood-prime. 
\end{proof}

The next result follows from \cite{dirac} and Theorem~\ref{Hamiltonian}.

\begin{corollary}
All graphs with minimum degree at least $n/2$ are neighborhood-prime. 
\end{corollary}

\begin{proof}
Assume there is a graph $G$ that is not neighborhood-prime with minimum degree at least $n/2$, so by \cite{dirac}, $G$ contains a Hamilton cycle. By Proposition~\ref{ham_max_edges}, since $G$ contains at least $\frac{n^2}{4}$ edges, $G$ is neighborhood-prime.
\end{proof}

\noindent There are many more similar results that can be made which follow from decades worth of results focusing on Hamiltonian graphs. See \cite{hamsurvey} for an excellent survey of the subject. 

By Theorem~\ref{Hamiltonian}, we know that all Hamiltonian graphs besides those of order $2$ modulo $4$ are neighborhood-prime, and so the upcoming work in Section~\ref{small} and our earlier results give rise to the following conjecture. 

\begin{conjecture}\label{ham_min_edges}
A Hamiltonian graph with more than $n+\frac{n}{3}$ edges is neighborhood-prime.
\end{conjecture}

If Conjecture~\ref{ham_min_edges} holds, then Figure~\ref{cong_sharp_example} would be a sharp example of a graph that is not neighborhood-prime with $n+n/3$ edges. A similar program to the one in Section~\ref{small} found at the first author's Github page\footnote{First author's Github page: \url{https://github.com/cecilartavion/neighborhood_prime_labeling}.} shows that it is not possible to place all of the even numbers from 1 to 18 inclusive on the vertices in Figure~\ref{cong_sharp_example} in such a way where the neighborhood of each vertex has at least one vertex that is not labeled. Thus there is no way to place an NPL on the vertices in Figure~\ref{cong_sharp_example} since there is at least one neighborhood that contains all even numbers. 

In our next result, we extend an observation (originally made in~\cite{PS2} on the maximum degree of $G$, denoted $\Delta(G)$. Notice that if $\Delta(G)=n-1$, then labeling a vertex $v$ with $\deg_G(v)=\Delta(G)$ by the label $1$ will ensure $G$ is neighborhood-prime. We can improve the degree requirement depending on the number of large prime numbers at our disposal. We use $\pi(x)$ for the number of prime numbers less than or equal to $x$.

\begin{figure}[htb]
\begin{center}
	\includegraphics[scale=1]{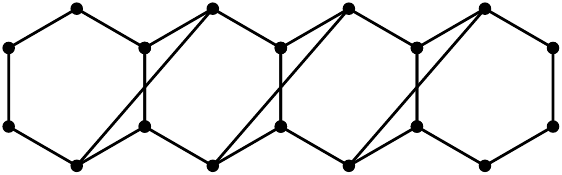}
\end{center}
\caption{Sharp example for Conjecture~\ref{ham_min_edges}}\label{cong_sharp_example}
\end{figure}

\begin{proposition}\label{large_degree}
Let $G$ be a graph of order $n\geq 6$.  If $\Delta(G)\geq n-\pi(n)+\pi(\lfloor n/2\rfloor)-1$, then $G$ is neighborhood-prime. 
\end{proposition}

\begin{proof}
Let $v$ be a vertex in $G$ with degree $\Delta(G)$, and let $u_1,\ldots, u_k$ be the vertices not in the neighborhood of $v$, where by our assumption $\displaystyle k\leq \pi(n)-\pi(\lfloor n/2\rfloor)$.  Let $p_1,\ldots, p_k$ be the first $k$ prime numbers in which $\displaystyle n/2<p_i\leq n$ for each $i$.  We assign a labeling $f: V(G)\rightarrow [n]$ as follows.  We first assign $f(v)=1$.  Then as we consider $i$ from $1$ to $k$, if $u_i$ has a neighbor $w_i$ that is not yet labeled, we assign $f(w_i)=p_i$.  

At this point, as many as $k$ of the neighbors of $v$ have been labeled.  If none of the vertices $w_i$ are in $N(v)$, then based on our assumption for the degree of $v$, this vertex has at least two unlabeled vertices $v'$ and $v''$ which we label with $2$ and $3$.  In the case of only one $w_i\in N(v)$, label another neighbor of $v$, say $v'$, with $2$.  If there are at least two $w_i\in N(v)$, we do not need to specify the labels of any other neighbors of $v$. All remaining unlabeled vertices in $G$ can be labeled distinctly in any way with the remaining integers in $[n]$.

To show this labeling described above is neighborhood-prime, let $x\in V(G)$ with $|N(x)|\geq 2$.  We need to consider the neighborhood of $x$ in three cases: $x\in N(v)$, $x=u_i$ for $i\in [k]$, or $x=v$.  In the first case where $x\in N(v)$, since $f(v)=1$, we have $\gcd\{f(N(x))\}=1$.

If $x=u_i$, then the prime number $p_i\in f(N(x))$ from the labeling of $w_i$, or $p_j\in f(N(x))$ with $j<i$ from the labeling of a previous vertex $w_j$.  Since $p_i> n/2$ (or likewise $p_j> n/2$), it will be relatively prime with any other labels that are adjacent to $u_i$.  Thus, $\gcd\{f(N(u_i))\}=1$.
  
In the final case of $x=v$, the set of labels $f(N(x))$ contains either $\{2,3\}$, $\{2,p_i\}$, or $\{p_i,p_j\}$ for some $i,j\in [k]$, depending on the cases described in the labeling.  In each situation, these pairs are relatively prime, making $\gcd\{f(N(v))\}=1$ and proving our labeling to be an NPL.
\end{proof}

Though it may not be clear which classes of graphs can be shown to be neighborhood-prime using only Proposition~\ref{large_degree}, there are indeed examples of such graphs where Proposition~\ref{large_degree} is the only result in this paper or in previously published papers on this topic that can show this graph is neighborhood-prime. Figure~\ref{example_npl} is one such example of a graph on $100$ vertices for which Proposition~\ref{large_degree} proves the graph is neighborhood-prime.

\begin{figure}[htb]
\begin{center}
\includegraphics[scale=1.5]{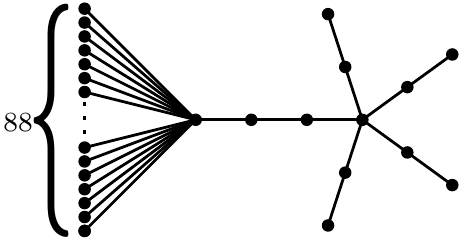}
\end{center}
\caption{Graph that is neighborhood-prime by Proposition~\ref{large_degree}}\label{example_npl}
\end{figure}

\section{NPLs For Graphs of Small Order}\label{small}

To justify some of our later conjectures, it is useful to classify as many graphs as possible with small order that are neighborhood-prime or not neighborhood-prime. We accomplish this task for all graphs with at most $10$ vertices by using Brenden McKay's repository of all graphs on at most $10$ vertices (not necessarily connected). 
The collection of these graphs can be found at Brenden McKay's website\footnote{Brenden McKay's webpage with graph data: \url{https://users.cecs.anu.edu.au/~bdm/data/graphs.html}.} or at the first author's Github page.\footnote{First author's Github page: \url{https://github.com/cecilartavion/neighborhood_prime_labeling}.}

In this section, we give an explanation of the python program used to prove Theorem~\ref{small_thm} before stating the main conclusion for this section. For the sake of completeness, the code for the program is provided in the appendix in Figure~\ref{code}. All comments have been removed from the code but can be found in the python file located at the first author's Github page.

Note that the files containing the graphs on 10 vertices was split using command-line input to ensure no graphs were lost in the process. This division of the file containing all graphs with exactly 10 vertices was necessary since not doing so would require more memory than a typical computer actually contains (more than 16Gb of RAM would be required). 
Currently, the program is set up to test all graphs with order between $2$ and $9$ vertices inclusive. To have the program run on the graphs with 10 vertices by making a quick change to file names being called, the program will check the graph with 10 vertices.

The program is fairly straightforward for readability purposes rather than computational speed. In the main module, the first for loop will go through each file, and build a set \texttt{g1} of the graphs to check. After \texttt{g1} is constructed, the second for loop will go through each graph and check each possible labeling to see if there is a labeling that would be an NPL. 
Note that the function \texttt{gen\_labelings(n)} will build all $n!$ possible labelings. The function \texttt{is\_np\_labeling} tests if the neighborhood of each vertex has $\gcd$ of $1$. If the program fails to find any NPL for a given graph, the program will print the number of vertices and edges, as well as the graph that is not neighborhood-prime. 
Since all possible labelings are checked, we can ensure that all the graphs found by the program are not neighborhood-prime and all other graphs on at most 10 vertices are neighborhood-prime. We summarize these findings in the following result and in Figure~\ref{non_np_graphs} in the appendix.

\begin{theorem}\label{small_thm}
The graphs in Figure~{\normalfont\ref{non_np_graphs}} are not neighborhood-prime, and all other graphs of order at most $10$ are neighborhood-prime.
\end{theorem}


\section{Random Graphs and Neighborhood-primality}\label{random}

For another perspective on neighborhood-prime labeling, we may ask whether a ``typical'' graph is neighborhood-prime.
One way to address this question is to sample the usual (Erd\H{o}s-R\'{e}nyi) random graph $G_{n,p}$ and find the probability that $G_{n,p}$ is neighborhood-prime.  
We use $G_{n,p}$ for the random graph on $n$ vertices in which edges (i.e. elements of $\binom{V}{2}$) are present independently each with probability $p$.
Before doing so, we will need the following which was shown by Bollob\'{a}s \cite{bollobas}, and Koml\'{o}s and Szemer\'{e}di \cite{KS}.
Recall that ``w.h.p.'' means with probability tending to one as $n$ tends to infinity.  Also, we say $f(n)=\omega (g(n))$ if $\lim_{n\rightarrow\infty} f(n)/g(n)=\infty$.

\begin{theorem}\label{thm:gnpham}
{\normalfont\cite{bollobas,KS}} If $p>(\ln n+\ln \ln n+g(n))/n$ where $g(n)=\omega (1)$, then $G_{n,p}$ is Hamiltonian w.h.p. 
\end{theorem}

The main takeaway from the following two theorems is simply that for $p>(\ln n+\ln \ln n+g(n))/n$, $\chi(G_{n,p})>2$ (where $\chi (G)$ is the chromatic number of $G$) w.h.p. and hence for such $p$, $G_{n,p}$ contains an odd cycle w.h.p.

\begin{theorem}\label{thm:chignpsmallp}
{\normalfont \cite{frieze}} If $p=\omega(1/n)$ and $p=o(1)$, then $\chi (G_{n,p})\geq 2 np/\ln (np)$ w.h.p. 
\end{theorem}

\begin{theorem}\label{thm:chignpbigp}
{\normalfont \cite{bollobas1988chromatic}} If $p=\Omega(1)$ and $b=1/(1-p)$, then $\chi (G_{n,p})\geq n/2 \log_{b} (n)$ w.h.p. 
\end{theorem}

With these results in hand, we obtain the following theorem.

\begin{theorem}\label{thm:gnpnbhdprime}
If $p>(\ln (n) + \ln \ln n + g(n))/n$ where $g(n)=\omega (1)$, then $G_{n,p}$ is neighborhood-prime w.h.p.
\end{theorem}

\begin{proof}
Let $p>(\ln n+\ln \ln n+g(n))/n$ where $g(n)=\omega (1)$.  
If $n\not\equiv 2\pmod{4}$, then $G_{n,p}$ is neighborhood-prime w.h.p. by Theorems~\ref{Hamiltonian} and \ref{thm:gnpham}.  
Recall that a graph with chromatic number greater than $2$ is not bipartite and thus contains an odd cycle. 
If $n\equiv 2\pmod{4}$, then the fact that $G_{n,p}$ is neighborhood-prime w.h.p. follows directly from Theorems \ref{thm:gnpham}, \ref{thm:chignpsmallp}, and Lemma~\ref{oddCycle} if $p=o(1)$ and from Theorems \ref{thm:gnpham}, \ref{thm:chignpbigp}, and Lemma~\ref{oddCycle} if $p=\Omega(1)$.
\end{proof}

The reader may have noticed that the (rough) heuristic in this section and in Section~\ref{restrict} is that graphs with lots of edges {\em should} be neighborhood-prime and so may be unsurprised by the previous theorem.  
Given this heuristic, it may be more natural to consider sparser random graphs.  
As such, let us now focus on $G_{n,d}$, the random $d$-regular graph (see \cite{wormald} for a survey of results on random regular graphs) for small $d$.  
As before we will need a couple of previous results.

\begin{theorem}\label{thm:gndham}
{\normalfont \cite{wormald}} If $d\geq 3$ is fixed, then $G_{n,d}$ is Hamiltonian w.h.p. 
\end{theorem}

\begin{theorem}\label{thm:indgnd}
{\normalfont \cite{mckay}} If $d\geq 3$ is fixed, then $\alpha (G_{n,d})<0.46n$ w.h.p. where $\alpha(G)$ is the independence number.
\end{theorem}

Together these give the following theorem for random regular graphs which is analogous to Theorem~\ref{thm:gnpnbhdprime}.

\begin{theorem}\label{thm:gndnbhdprime}
If $d\geq 3$ is fixed, then $G_{n,d}$ is neighborhood-prime w.h.p.
\end{theorem}

\begin{proof}
Let $d\geq 3$ be fixed.  
If $n\not\equiv 2\pmod{4}$, then $G_{n,d}$ is neighborhood-prime w.h.p. by Theorems~\ref{Hamiltonian} and \ref{thm:gndham}.  
If $n\equiv 2\pmod{4}$, then the fact that $G_{n,d}$ is neighborhood-prime w.h.p. follows directly from Theorems \ref{thm:gndham}, \ref{thm:indgnd}, and Lemma~\ref{oddCycle} using that $\chi (G) \geq n/\alpha (G)$, which implies $\chi (G_{n,d}) \geq n/0.46n > 2$.
\end{proof}

It is worth pointing out that having a neighborhood-prime labeling is nearly an {\em increasing graph property}.  An increasing graph property is a graph property which is closed under addition of edges; that is, if $G$ satisfies a property $\mathcal{P}$ and $e\not\in G$, then $G + e$ satisfies $\mathcal{P}$ as well.  Having a neighborhood-prime labeling nearly satisfies closure under edge addition with edges on degree one vertices being the only exception.  That is, if $G$ is neighborhood-prime, $\{x,y\}\not\in G$, and $d(x),d(y)>1$, then $G +  \{x,y\}$ is neighborhood-prime.  In light of several of the results in this paper and the fact that the graphs in Figure~\ref{non_np_graphs} are not neighborhood-prime, it seems as if satisfying a (rather weak) local density condition may be enough to ensure that a graph is neighborhood-prime.  In addition, Cloys and Fox \cite{Fox} showed that all trees with no degree two vertices are neighborhood-prime. Thus we propose the following conjecture.

\begin{conjecture}
If $G$ has $\delta(G)\geq 3$, then $G$ is neighborhood-prime. 
\end{conjecture}

\section{Neighborhood Graphs and Trees}\label{sec:trees}

Given a graph $G$, we define the set of \textit{neighborhood graphs} of $G$, denoted $\mathcal{N}(G)$, as the set of all possible graphs $H$ such that $V(H)=V(G)$, and for each $v\in V(G)$ with $\deg(v)\geq 2$, there exists exactly one edge $uw\in E(H)$ where $u,w\in N(v)$.  Note that when every vertex in $G$ is degree 1 or 2, there is a unique neighborhood graph, whereas a vertex of larger degree results in multiple choices for which pair of vertices in its neighborhood is included in $E(H)$.

\begin{example}
{\rm 
Given a path $P_n$, $\mathcal{N}(P_n)$ consists solely of the disjoint union of paths $P_{n/2}\cup P_{n/2}$ if $n$ is even or $P_{(n+1)/2}\cup P_{(n-1)/2}$ if $n$ is odd.
Given a cycle $C_n$, $\mathcal{N}(C_n)$ contains only the graph $C_{n/2}\cup C_{n/2}$ if $n$ is even or simply $C_n$ if $n$ is odd.
}
\end{example}

A neighborhood graph of $G$ can be useful for developing an NPL on $G$ due to the following connection between neighborhood-prime and prime labelings.

\begin{theorem}\label{neighborhoodGraph}
For any graph $G$, if there exists an $H\in\mathcal{N}(G)$ with a prime labeling, then $G$ is neighborhood-prime.
\end{theorem}
\begin{proof}
Assume $f:V(H)\rightarrow [|V(H)|]$ is a prime labeling of $H$.  Let $v\in V(G)$ with $\deg{v}\geq 2$, and suppose $u,w$ are the neighbors of $v$ for which $uw$ is an edge in $H$.  Since $f$ is a prime labeling of $H$, we have $\gcd\{f(u),f(w)\}=1$, and thus $\gcd\{f(N(v))\}=1$ as well because $\{f(u),f(w)\}\subseteq f(N(v))$.
\end{proof}

\begin{theorem}\label{2regular}
Given a $2$-regular graph $G$, $H\in \mathcal{N}(G)$ has a prime labeling if and only if $G$ is neighborhood-prime.
\end{theorem}
\begin{proof}
The direction in which a prime labeling implies a neighborhood-prime labeling is proven by Theorem~\ref{neighborhoodGraph}.  For the other direction, assume $f: V(G)\rightarrow [|V(G)|]$ is an NPL of $G$.  The fact that $G$ is $2$-regular implies each vertex $v$ has a neighborhood $N(v)=\{u,w\}$ in which $\gcd\{f(u),f(w)\}=1$.  This results in the relatively prime condition being satisfied for each edge $uw\in E(H)$, making $f$ a prime labeling of $H$.
\end{proof}

The only $2$-regular graphs are cycles or unions of cycles, and cycles have been shown in~\cite{PS2} to be neighborhood-prime if $n\not\equiv 2\pmod{4}$.  Neighborhood-prime labelings for the disjoint union of two cycles were investigated by Patel and Shrimali in~\cite{PS3}.  They determined $C_n\cup C_m$ is neighborhood-prime if $m$ is odd and $n\equiv 0\pmod{4}$ or if $m,n\equiv 0\pmod{4}$.  Meanwhile, they further demonstrated the union of two disjoint cycles is not neighborhood-prime when $m$ and $n$ are both odd, $m$ is odd and $n\equiv 2\pmod{4}$, $m,n\equiv 2\pmod{4}$, or $m\equiv 2\pmod{4}$ and $n\equiv 0\pmod{4}$.
 Theorem~\ref{2regular} provides alternative explanations for these neighborhood-prime results since we can rely on the larger body of knowledge on prime labelings.  For example, Deretsky et al.\ showed in~\cite{DLM} that any union of cycles with at least two odd cycles is not prime, and for each union of cycles $C_n\cup C_m$ shown in~\cite{PS3} to not be neighborhood-prime, the neighborhood graph in $\mathcal{N}(C_n\cup C_m)$ includes the union of two or more odd cycles. 
On the other hand, the unions $C_{2k}\cup C_n$ is prime for all integers $k$ and $n$ as well as the union of three or four cycles with at most one odd cycle are prime.  Patel~\cite{Patel2} extended these results to introduce a prime labeling for $C_{2k}\cup C_{2k}\cup C_{2k} \cup C_{2m} \cup C_n$.  It was conjectured in~\cite{DLM} that all unions of cycles with at most one odd cycle are prime.  Since $\mathcal{N}(C_{4k})$ only includes the union of even cycles $C_{2k}\cup C_{2k}$ and $\mathcal{N}(C_{2k+1})$ consists of the same cycle $C_{2k+1}$, it follows that an affirmative result for the conjecture by Deretsky et al., combined with Theorem~\ref{2regular}, would imply the following is also true.

\begin{conjecture}
The union of cycles $C_{4k_1}\cup C_{4k_2}\cup\cdots \cup C_{4k_{\ell}}\cup C_n$ is neighborhood-prime if $n\equiv 0\pmod{4}$ or if $n$ is odd.
\end{conjecture}

Very few classes of graphs have been shown to not be neighborhood-prime thus far in papers published on this particular graph labeling.  Combining results by Deretsky et al.\ with Theorem~\ref{2regular} once again, we can show quite a large class of unions of cycles do not have an NPL.

\begin{proposition}
The graph $C_{n_1}\cup \cdots \cup C_{n_k}$ is not neighborhood-prime if $n_i\equiv 2\pmod{4}$ for any $i\in[k]$ or if at least two of the cycles are of odd length.
\end{proposition}

\begin{proof}
Since the given union of cycles is $2$-regular, $G$ has a unique neighborhood graph which is also a union of cycles.  If a cycle $C_{n_i}$ satisfies $n_i\equiv 2\pmod{4}$, then it has $C_{n_i/2}\cup C_{n_i/2}$ as its neighborhood graph where $n_i/2$ is odd.  Likewise, the case of the union containing two odd cycles would result in them remaining odd cycles within the union that is the neighborhood graph of $G$.  In either case, $\mathcal{N}(G)$ is a union consisting of at least two odd cycles, which is not prime by Theorem~5 in~\cite{DLM}.  Thus, by Theorem~\ref{2regular}, $G$ is not neighborhood-prime.
\end{proof}

We now shift our focus to investigate neighborhood-prime labelings of trees.  Cloys and Fox~\cite{Fox} made an analogous conjecture to Entriger's tree primality conjecture that all trees are neighborhood-prime.  Despite proving NPLs exist for many classes of trees, such as caterpillars, spiders, and firecrackers, or degree $2$, vertices, they were unable to show an NPL exists for the class of trees known as \textit{lobsters}.  This tree consists of a path called the spine in which each vertex is within distance $2$ from the spine. 

We will make use of the fact that if $G$ is neighborhood-prime, then adding a single pendant vertex to a non-pendant vertex means the resulting graph is still neighborhood-prime. This is summarized in the following observation which is a consequence of repeated applications of a result by Cloys and Fox~\cite{Fox}.

\begin{obs}\label{pendantAdd} 
If $H$ is neighborhood-prime then the graph $G$ formed by the addition of any number of pendant vertices adjacent to vertices of degree greater than $2$ in $H$ is neighborhood-prime. 
\end{obs}

The next result proves that all lobsters are neighborhood-prime if it is true that all caterpillars are prime.  A \textit{caterpillar} is a tree with all vertices being within distance 1 of a central path. This prime labeling result was claimed to be true in the graph labeling dynamic survey by Gallian~\cite{Gallian} before the $20^{\rm th}$ edition, and is cited in~\cite{FH} and~\cite{LWY} as being proven in a preprint by Acharya that remains unpublished.  
Partial results that work towards caterpillars being prime include Tout et al~\cite{TDH}, who proved caterpillars are prime if every spine vertex has the same degree or if the degree of each spine is at most $5$.  However, there appears to be no published result proving that all caterpillars are prime.  

\begin{theorem}
If all caterpillars are prime, then all lobster graphs are neighborhood-prime. 
\end{theorem}

\begin{proof}
Suppose all caterpillar graphs are prime. Then all forests that are two disjoint caterpillar graphs are also prime since the deletion of an edge between two vertices in the spine of a caterpillar graph will form two disjoint caterpillar graphs.

By Observation~\ref{pendantAdd}, we need only consider lobsters whose non-spine vertices have degree $1$ or $2$. In fact, all of the vertices that are not in the spine but are adjacent to the spine must be degree 2, otherwise we can delete pendants from these vertices or from the spine using Observation~\ref{pendantAdd} without changing whether the graph is or is not neighborhood-prime; let $L$ be such a lobster graph. 

Let the spine of $L$ have $n$ vertices with degree $d_1,d_2,\ldots,d_n$.   We will partition these $n$ degrees into two sets depending on the parity of the indices.  Since the cases of $n$ being even or odd will follow analogously, we assume without loss of generality that $n$ is even, resulting in a partition of the set of degrees into $\{d_1,d_3,\ldots,d_{n-1}\}$
 and $\{d_2,d_4,\ldots,d_n\}$.  Consider two caterpillar graphs $D_1$ and $D_2$ with spines of length $\frac{n}{2}$ and $\frac{n}{2}$. 
Let $D_1$ have vertices with degrees $d_1,d_3,d_5,\ldots,d_{n-1}$ on the spine and $D_2$ have degrees $d_2,d_4,d_6,\ldots,d_n$ on the spine. By assumption, the disjoint union of $D_1$ and $D_2$ is prime, so place such a prime labeling on the vertices in $D_1$ and $D_2$. Since a forest of disjoint caterpillar graphs (with additional isolated vertices that can easily be labeled) is in $\mathcal{N}(L)$, by Theorem~\ref{neighborhoodGraph}, $L$ is neighborhood-prime.
\end{proof}

Since a published proof of all caterpillars being prime cannot be found, we must assume that such a result does not exist, and we add the following conjecture.

\begin{conjecture}
All lobsters are neighborhood-prime. 
\end{conjecture}

We continue this investigation of lobsters by considering ones with particular structures or restrictions.  
In Proposition~\ref{lobster_result}, we will make use of a number theory result by Pomerance and Selfridge~\cite{PS}.

\begin{theorem}\label{selfridge}
{\normalfont\cite{PS}} If $N$ is a natural number and $I$ is an interval of $N$ consecutive integers, then there is a one-to-one correspondence $f:[N]\to I$ such that $\gcd\{i,f(i)\}=1$ for $1\leq i\leq N$.
\end{theorem} 

\begin{proposition}\label{lobster_result}
Let $L$ be a lobster where $v_1,v_2,\ldots,v_n$ are the vertices on the spine of $L$ with degree $d_1,d_2,\ldots,d_n$, respectively. Let $d_{n+1},d_{n+2},\ldots,d_{n+k}$ be the degrees of the non-pendant vertices that are not on the spine but are adjacent to a vertex on the spine. Let $d_1',d_2',\ldots,d_n'$ be the number of non-pendant vertices adjacent to $v_1,v_2,\ldots,v_n$, respectively, and let $d'=\max\{d_1',d_2',\ldots,d_n'\}$. 
If 
\[
\sum_{i=1}^{n+k}(d_i-2)+2 \geq \sum_{i=1}^n (d'-d_i'),
\]
then $L$ is neighborhood-prime. 
\end{proposition}

\begin{proof}
We will begin defining an NPL of $L$ by labeling the spine as has been typically done (consecutively using every other vertex and similar to the labeling in Equation~\eqref{cycle}) so that $\{f(v_1),\ldots,f(v_n)\}=[n]$ and the greatest common divisor of the labeling on the neighborhood of $v_i$ is $1$ for all $i\in \{ 2, \dots , n-1 \}$ (and recall that $d(v_1)=d(v_n)=1$).
Let $u_{1},u_{2},\ldots,u_{k}$ be the vertices described above with degrees $d_{n+1},d_{n+2},\ldots,d_{n+k}$ respectively. 
Let $S$ be the set of vertices adjacent to the spine or, for each $j\in [k]$, the set of vertices that includes all but one of the pendant vertices to $u_j$. Note that if we can label exactly one of the pendant vertices to $u_j$ so that the greatest common divisor of the label on the one pendant vertex and the spine is $1$ by Theorem~\ref{selfridge} (as described below), we can put any label on the remaining pendant vertices to $u_j$. Also, note that 
\[
|S|=\sum_{i=1}^n (d_i-2)+2+\sum_{i=n+1}^{n+k} (d_i-2) = \sum_{i=1}^{n-k} (d_i-2)+2.
\]
Vertices in $S$ are {\em surplus} vertices in the sense that by carefully labeling $V(L)\setminus S$ we will already have a neighborhood-prime labeling of $L$ and so vertices in $S$ may receive any of the remaining labels.
Using Theorem~\ref{selfridge}, there is a function $g:[n]\to \{n+1,\ldots,2n\}$ such that for each $i\in [n]$, if $v_i$ is adjacent to a $u_j$ with $\deg(u_j)\geq 2$, we can label exactly one pendant vertex of $u_j$ with $g(i)$. 
If there does not exist a $u_j$ for $v_i$, then we will label one of the vertices in $S$ with this label. Apply Theorem~\ref{selfridge} a total of $d'$ times. Since we label $d'-d_i'$ vertices in $S$ for each $v_i$ and $|S|\geq \sum_{i=1}^n d'-d_i'$, it is clear that we can label vertices in $L$ so that $\gcd\{f(N(u_j))\}=1$ for all $u_j\in \{u_1,\ldots,u_k\}$. 
The remaining vertices can be labeled in any manner, resulting in $f$ being an NPL. 
\end{proof}

One particular type of lobsters to consider is the \textit{reduced lobster}, which is a lobster with no pendant vertices adjacent to the spine and each non-spine neighbor of a vertex in the spine is degree $2$.  We will also consider our lobsters from this point to have pendants at the ends of the spine, else we could extend our spine to include more vertices.  In order to show this set of lobster graphs with certain degree restrictions is neighborhood-prime, we first consider how to create a prime labeling for unions of stars of certain sizes, which are examples of a neighborhood graph of this particular type of lobster.  Youssef and Elsakhawi~\cite{EY} demonstrated that the union of any two stars is prime, whereas we will consider any number of stars where we limit the degree of the vertices.

To label each star, we will use a number theoretical result by Pillai~\cite{pillai}.  
He proved that for any sequence of $m$ consecutive integers with $m\leq 16$, there exists an integer $x$ in the sequence such that $\gcd\{x,y\}=1$ for all other $y$ in the sequence.  This result allows us to label each star in a union of stars with consecutive integers with the number $x$ being used as the center of that star.  See Figure~\ref{union_of_stars} for an example of $S_{15}\cup S_8\cup S_5\cup S_4 \cup S_1$ with a prime labeling where $S_i$ is a star with $i$ pendant vertices. 

\begin{theorem}\label{unionOfStars}
The union of stars $S_{i_1}\cup \cdots \cup S_{i_n}$ is prime if at most one star $S_{i_j}$ has $i_j>15$.
\end{theorem}
\begin{proof}
If there is a star $S_{i_j}$ with $i_j>15$, then we label the center of the star as $1$ and the remaining leaves as $2,\ldots, i_j+1$, which trivially makes the endpoints of each edge in this star relatively prime.

We proceed to label the remaining stars in any order but where each star's $i_k+1$ vertices are labeled by integers in an interval $[m,m+i_k]$ of the smallest available labels.  Given that $i_k\leq 15$, there exists an integer $x\in [m,m+i_k]$ in which $x$ is relatively prime with all primes $p\leq i_k$ based on this interval containing at most $16$ integers, according to the property proven by Pillai~\cite{pillai}.  Thus, labeling the center of $S_{i_k}$ with $x$ and the leaves with the remaining integers from the interval would fulfill the prime labeling condition and results in $S_{i_1}\cup \cdots \cup S_{i_n}$ being prime.
\end{proof}

\begin{figure}[htb]
\begin{center}
\includegraphics[scale=1]{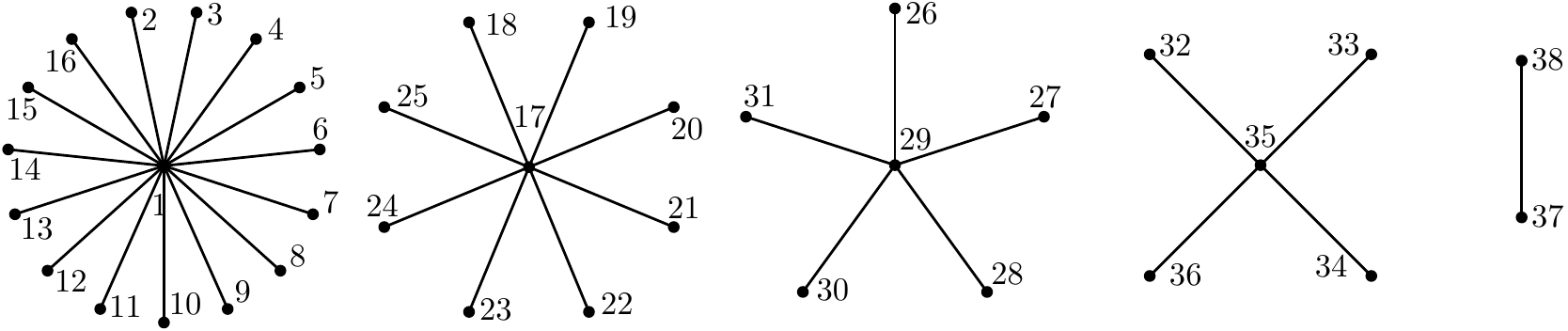}
\end{center}
\caption{Union of stars $S_{15}\cup S_8\cup S_5\cup S_4 \cup S_1$}\label{union_of_stars}
\end{figure}

\begin{figure}[htb]
\begin{center}
\includegraphics[scale=1]{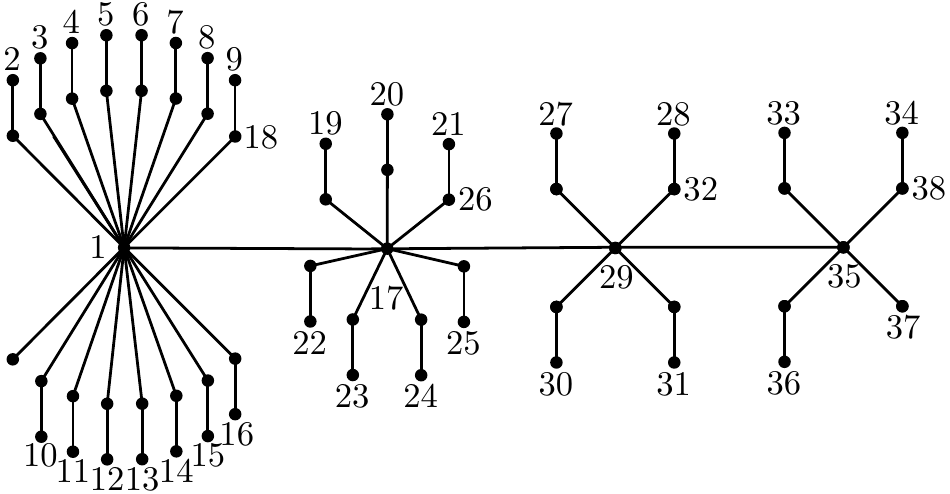}
\end{center}
\caption{Lobster with degree $17$, $9$, $6$, and $5$ on the interior vertices of the spine.}\label{lobster_ex}
\end{figure}

In order to determine if reduced lobsters are neighborhood-prime, we will try to find a prime labeling of its neighborhood graph, one of which is described in the following proposition.  See Figure~\ref{lobster_ex} for an example of a reduced lobster and part of an NPL that is created based on the prime labeling of its neighborhood graph displayed in Figure~\ref{union_of_stars}.  The remaining vertices can be labeled in any way by distinct integers in $\{39,\ldots, 64\}$ to complete the NPL. 

\begin{proposition}
All reduced lobsters in which the interior vertices $v$ on the central path satisfy $3\leq \deg(v)\leq 16$, except possibly for one vertex $v'$ with degree exceeding $16$, are neighborhood-prime.
\end{proposition}
\begin{proof}
We consider the neighborhood graph of a lobster $L$ that satisfies the assumed degree requirements.  We call the spine vertices $v_1,\ldots, v_n$ for some $n\in\mathbb{N}$.  For each interior spine vertex $v_i$ for $i=2,\ldots, n-1$, we will consider $u_{i,1},\ldots, u_{i,i_k}$ to be its neighbors not on the central path where $1\leq i_k\leq 14$ for each $v_i\neq v'$ and $i_k\geq 1$ for $v_i=v'$.  We then refer to the leaf that is adjacent to $u_{i,j}$ as $w_{i,j}$.  One particular graph $H\in\mathcal{N}(L)$ exists in which we choose the following edges to include for each non-leaf's neighborhood.

For each $u_{i,j}$, the edge in $H$ using vertices in $N(u_{i,j})$ is forced to be $w_{i,j}v_i$.  For every $v_i$ with $i=2,\ldots, n-1$, since its degree is at least $3$, we choose from its neighborhood an edge $u_{i,1}v_{i+1}$.  The resulting graph $H$ is a disjoint union of stars (with additional isolated vertices) in which the vertex $v_2$ is a star $S_{m_2}$ where $m_2=\deg(v_2)-2$ since the only edges in this star are $w_{2,j}v_2$.  There will also be a star $S_{m_i}$ with center $v_i$ for $i=3,\ldots, n-1$ where in this case $m=\deg(v_i)-1$ since the edge $u_{i-1,1}v_i$ is also included.  Finally, there is a star $S_2$ formed only by the edge $u_{n-1,1}v_n$.  

We see that the graph $H$ contains a union of stars in which each star has at most $\deg(v_i)\leq 16$, other than potentially one vertex $v'$ with higher degree.  By Theorem~\ref{unionOfStars}, this union of stars has a prime labeling, and any isolated vertices in $H$ can be labeled using the remaining unused labels from $[|V(L)|]$.  Thus Theorem~\ref{neighborhoodGraph} proves that this type of lobster graph is neighborhood-prime.
\end{proof}

Note that while the structure of a reduced lobster is helpful in specifying the degree requirements and describing the neighborhood graph, Observation~\ref{pendantAdd} extends this previous result to lobsters that are not reduced by adding any number of pendants to the interior spine vertices or to their neighbors $u_{i,j}$.

Cloys and Fox \cite{Fox} conjectured that all trees are neighborhood-prime. Though our partial results on lobsters are not enough to prove this conjecture, we can show that if the order of the tree is large enough, all such trees are neighborhood-prime. This can be done with the following result from Haxell, Pikhurko, and Taraz~\cite{HPT}. Note that it is stated in \cite{HPT} that estimates have that $n'>10^{10^{100}}$ would be an order that is large enough for Theorem~\ref{haxell} to hold.

\begin{theorem}{\normalfont \cite{HPT}}\label{haxell}
There exists an $n'$ such that every tree with $n\geq n'$ vertices is prime. 
\end{theorem}

Note that if $T'\in \mathcal{N}(T)$ where $T$ is a tree, then $|V(T')|>n/2$ where $|V(T)|=n$. Since for any tree $T$, the neighborhood-prime graph of $T$ is a forest, the following result is a consequence of Theorem~\ref{neighborhoodGraph} and since a prime graph is still prime after having an edge deleted. 

\begin{theorem}
There exists an $n'$ such that every tree with $n\geq n'$ vertices is neighborhood-prime. 
\end{theorem}

\bibliographystyle{amsplain}
\bibliography{vdec}

\newpage

\section*{Appendix}

\begin{figure}[htb]
\begin{lstlisting}
import numpy as np
import matplotlib.pyplot as plt
import networkx as nx
from math import gcd
import itertools
from functools import reduce

def gen_labelings(n):
    lst = list(itertools.permutations([x+1 for x in range(n)]))
    return lst

def is_np_labeling(G,labeling):
    found_label = True
    for v in nx.nodes(G):
        new_lst = [lst[x] for x in G.neighbors(v)]
        if len(new_lst) > 1: 
            if reduce(gcd, new_lst) != 1:
                found_label = False
    return found_label

if __name__ == "__main__":
    for N in range(2,10):
        g1 = nx.read_graph6("graph{}.g6".format(N))
        lsts = gen_labelings(len(g1[0].nodes()))
        for g in g1:
            one_good_lst = False
            for lst in lsts:
                if is_np_labeling(g,lst):
                    one_good_lst = True
                    break
            if not one_good_lst:
                print('Graph with {} vertices and {} edges.'
                      .format(len(g.nodes()),len(g.edges())))
                fig, ax = plt.subplots(1,1)
                nx.draw_networkx(g, with_labels=True, node_size = 50, node_color='orange',font_size=10,ax=ax)
                plt.axis('off')
                plt.show()
\end{lstlisting}
\caption{Code used to determine which small order graphs are neighborhood-prime}\label{code}
\end{figure}

\begin{figure}[htb]
\begin{center}
	\includegraphics[scale=.8]{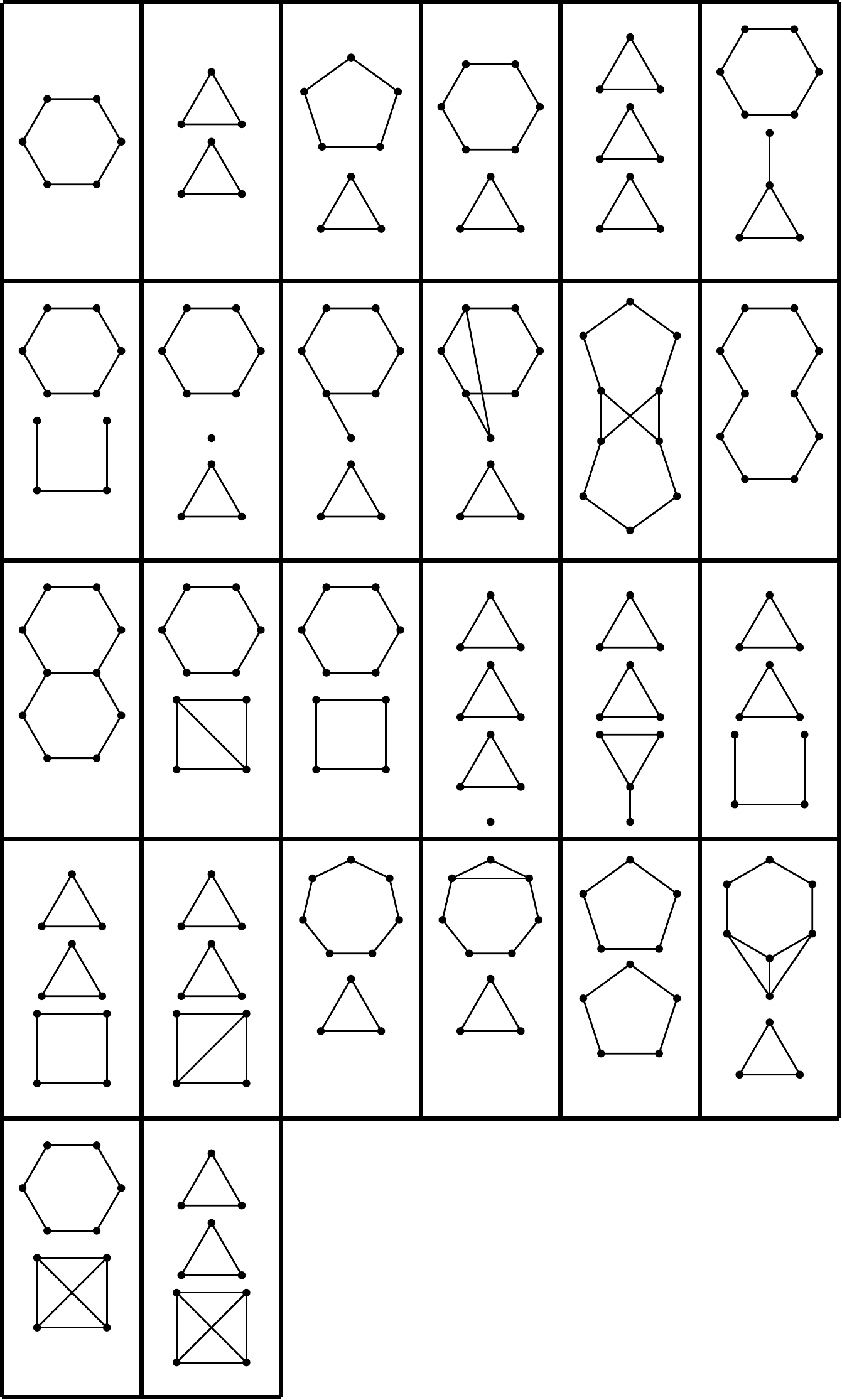}
\end{center}
\caption{All graphs on at most $10$ vertices with no neighborhood-prime labeling}\label{non_np_graphs}
\end{figure}

\end{document}